\documentclass[a4paper,11pt]{article}

\usepackage{amsmath,amsthm,amsfonts,amssymb,graphics,epsfig,color}

\usepackage[left=2cm,top=3cm,right=2cm,bottom=3cm,bindingoffset=0.5cm]{geometry}
\usepackage[latin1]{inputenc}
\usepackage[english]{babel}
\usepackage{verbatim}
\usepackage{amscd,enumerate}
\usepackage{epstopdf}
\usepackage{graphicx}
\usepackage{multicol}

\newtheorem{teo}{Theorem}[section]

\newtheorem{lema}{Lemma}[section]

\newtheorem{obs}{Remark}[section]

\newcommand{\dee}{\mathop{\mathrm{d}\!}}

\newcommand\be{\begin{equation}}
\newcommand\ee{\end{equation}}
\DeclareMathOperator\erf{erf}
\DeclareMathOperator\erfc{erfc}

\begin{document}

\title{Existence and uniqueness of the p-generalized modified error function}
\author{
Julieta Bollati$^{1,2}$, Jos\'e A. Semitiel$^{2}$, Mar\'ia F. Natale$^{2}$,\\
Domingo A. Tarzia $^{1,2}$ \\
\small {{$^1$} CONICET, Argentina} \\
\small {{$^2$} Depto. Matem\'atica, FCE, Univ. Austral, Paraguay 1950} \\
\small {S2000FZF Rosario, Argentina.}\\
\small{Email: jbollati@austral.edu.ar, jsemitiel@austral.edu.ar,}\\ \small{fnatale@austral.edu.ar, dtarzia@austral.edu.ar }
}
\date{}

\maketitle
\abstract{
In this paper,  the p-generalized modified error function is defined as the solution to a non-linear ordinary differential problem of second order with a Robin type condition at $x=0$.
Existence and uniqueness of a non-negative $C^{\infty}$  solution is proved by using a fixed point strategy. It is shown that the p-generalized modified error function converges to the p-modified error function defined as the  solution to a similar problem with a Dirichlet condition at $x=0$.
In both problems, for $p=1$, the generalized modified error function and the modified error function, studied recently in literature, are recovered. In addition, existence and uniqueness of solution to a problem with a Neumann condition is also analysed.
}


\smallskip


\section{Introduction.}

In \cite{CeSaTa18}, it was studied a fusion Stefan problem with variable thermal conductivity and a Robin boundary condition at the fixed face $x=0$
given by:

\begin{align}
& \rho c \frac{\partial T}{\partial t}=\frac{\partial}{\partial x} \left(k(T)\frac{\partial T}{\partial x}\right),& 0<x<s(t), \quad t>0, \label{EcCalor}\\
& k(T(0,t)) \frac{\partial T}{\partial x}(0,t)=\frac{h}{\sqrt{t}}\left[T(0,t)-T_{_0}\right], &t>0, \label{CondBordeConv}\\
&  T(s(t),t)=T_{_f}, &t>0, \label{TempCambioFase}\\
&  k\left(T(s(t),t)\right)\frac{\partial T}{\partial x}(s(t),t)=-\rho l \dot s(t), &t>0, \label{CondStefan}\\
& s(0)=0,\label{FrontInicial}
\end{align}
where the unknown functions are the temperature $T$ and the free boundary $s$ separating both phases. The parameters $\rho>0$ (density), $l>0$ (latent heat per unit mass), $T_f$ (phase-change temperature),  $T_{0}>T_f$ (bulk temperature), $h>0$ (coefficient that characterizes the heat transfer at $x=0$) and $c$ (specific heat) are all known constants. The problem  (\ref{EcCalor})-(\ref{FrontInicial})  is a phase-change problem known in the literature as a Stefan problem. It corresponds to the melting of a semi-infinite material which is initially solid at the phase-change temperature $T_f$. As $T_0 >T_f$, a phase-change interface $x=s(t)$, $t>0$ is beginning at $t=0$ with the initial position $s(0)=0$. Then, the temperature of the liquid phase is $T=T(x,t)$ defined in the domain $0<x<s(t)$, $t>0$, and the temperature of the solid phase is $T=0$ defined in the domain $x>s(t),t>0$.

 In  \cite{ChSu74},  the thermal conductivity $k$ is defined as :
\begin{align}
&k(T)=k_{0}\left(1+\delta\left(\frac{T-T_{f}}{T_{0}-T_f}\right)\right)\label{kAndre},
\end{align}
where $\delta$ is a given positive constant and $k_{0}$ is the reference thermal conductivity.
The existence of solution to problem (\ref{EcCalor})-(\ref{FrontInicial}) when the thermal conductivity $k(T)$ is defined by (\ref{kAndre}) has been proved through the existence of  what the authors  in \cite{CeSaTa18} called a {\em generalized modified error function} (GME), which is defined as the solution to the following ordinary differential problem
\begin{subequations}\label{Pb:Chicas}
\begin{align}
\label{eq:y}&[(1+\delta y(x))y'(x)]'+2xy'(x)=0\quad 0<x<+\infty\\
\label{cond:0}&\left(1+\delta y(0)\right)y'(0)-\gamma y(0)=0\\
\label{cond:infty}&y(+\infty)=1
\end{align}
\end{subequations}
where
\begin{equation}
\gamma=2\mathrm{Bi},\qquad
 \mathrm{with\qquad Bi}=\frac{h \sqrt{\alpha_0}}{k_0} \quad
 \mathrm{(generalized\, Biot\, number)}
\end{equation}
and
\begin{equation}
\alpha_0=\frac{k_0}{\rho c} \quad \mathrm{(thermal\, diffusivity)}.
\end{equation}

The solution of the problem  (\ref{EcCalor})-(\ref{FrontInicial}) is given as a function of the solution of the ordinary differential  problem (\ref{Pb:Chicas}) through the similarity variable $\frac{x}{2\sqrt{\alpha_0t}}$, see  \cite{CeSaTa18,ChSu74, OlSu87}. More explanation is given in \cite{AlSo,Gu,Ta11}

Motivated by \cite{KuRa18} we define a generalized thermal conductivity as:
\begin{align}
&k(T)=k_{0}\left(1+\delta\left(\frac{T-T_{f}}{T_{0}-T_f}\right)^p\right), \quad p \geq 1\label{k}.
\end{align}

Therefore, the existence of solution to problem (\ref{EcCalor})-(\ref{FrontInicial}) with $k$ given by (\ref{k}) will be studied through the  {\em p-generalized modified error function} (p-GME) which we define as  the solution to the following nonlinear differential problem:
\begin{subequations}\label{Pb:c}
\begin{align}
\label{eq:y}&[(1+\delta y^p(x))y'(x)]'+2xy'(x)=0\quad 0<x<+\infty\\
\label{cond:0}&\left(1+\delta y^p(0)\right)y'(0)-\gamma y(0)=0\\
\label{cond:infty}&y(+\infty)=1
\end{align}
\end{subequations}

Observe that in case  $p=1$ we recover the problem studied in \cite{CeSaTa17,CeSaTa18} originally defined in \cite{ChSu74,OlSu87}. Others papers in the subject for $p=1$ can be found  in \cite{Bo17, SaTa11}. In that sense, the p-GME function constitutes a mathematical generalization of the GME function.

With the purpose to prove existence and uniqueness of the p-GME function, i.e. a solution to problem (\ref{Pb:c}), in Section 2 we  define a  convenient contracting mapping.

In Section 3,   we study the asymptotic behaviour of the p-GME function when $\gamma\to\infty$. We will show that this function converges to the solution to an ordinary differential problem that arises by changing the Robin condition at $x=0$ \cite{CaJa59} by a Dirichlet one.

Finally, in Section 4 we change the Robin condition by a Neumann one in a solidification process and analyse the existence and uniqueness of a new ordinary differential problem.

In conclusion, the aim of this paper is to prove existence and uniqueness of solution to three ordinary differential problems that have been motivated from Stefan problems, imposing different boundary conditions at the fixed face $x=0$: Robin, Dirichlet and Neumann.

\section{Existence and uniqueness of solution to the p-GME function}\label{Sec:EyU}

Let us define:



\begin{align}
\label{X}&X=\left\{h:\mathbb{R}_0^+\to\mathbb{R}\,/\,h\text{ is a continuous and bounded real-valued function}\right\}\\
\label{K}&K=\left\{h\in X\,/\, ||h||_{\infty}\leq 1,\,0 \leq h,\,h(+\infty)=1\right\}.
\end{align}

We remark that $K$ is a non-empty closed convex and bounded subset of the Banach space $X$ by the usual norm
\begin{equation*}
  ||h||_{\infty}=\sup\limits_{x\in\mathbb{R}_0^+}\vert h(x)\vert<\infty,
\end{equation*}

 e.g.  \cite{Dh09} (page 2487), \cite{Dh10} (page 152), \cite{My48} (page 132) .


In this Section we will prove existence and uniqueness of the p-GME function (problem (\ref{Pb:c})) by using the Banach fixed point theorem. First, we will show that the ordinary differential problem  (\ref{Pb:c}) becomes equivalent to an integral equation.
We consider that  $\gamma$ is a parameter for problem (\ref{Pb:c}) and in Section 3 we will  study the  asymptotic behaviour when $\gamma\to\infty$.

\begin{teo}\label{Th:caract1}
Let  $\delta\geq 0$, $\gamma>0$, $p\geq 1$. For each $\gamma>0$, the function $y_\gamma\in K$ is a solution to problem (\ref{Pb:c}) if and only if $y_\gamma$  is a fixed point of the  operator $T_{\gamma}: K \rightarrow K$  given  by:
\begin{equation}\label{op:T}
T_{\gamma}(h)(x)=\frac{1+\gamma \displaystyle\int_0^{x} f_h(\eta) \mathrm{d}\eta}{1+\gamma \displaystyle\int_0^{\infty} f_h(\eta) \mathrm{d}\eta},\qquad x\geq 0,
\end{equation}
with
\begin{equation}\label{fh}
f_h(x)=\frac{1}{\Psi_h(x)}\exp\left(-2\displaystyle\int_0^x\frac{\xi}{\Psi_h(\xi)}\mathrm{d}\xi\right),\qquad \Psi_h(x)=1+\delta h^{p}(x).
\end{equation}
\end{teo}

\begin{proof}  Notice first that for each  $y=y_\gamma\in K$ we can easily obtain
\begin{equation}\label{cotafh}
\dfrac{\exp(-\eta^2)}{1+\delta}\leq f_y(\eta)\leq  \exp\left({-\dfrac{\eta^2}{1+\delta}}\right),
\end{equation}
from where it follows that
\begin{equation}\label{cotadef}
0<\dfrac{\gamma \sqrt{\pi}}{2(1+\delta)}<  1+ \gamma \displaystyle\int_0^{\infty} f_y(\eta) \mathrm{d}\eta\leq 1+ \dfrac{\gamma \sqrt{1+\delta} \sqrt{\pi}}{2}.
\end{equation}

Taking into account (\ref{cotadef}) we have $T_\gamma(y)$ is a continuous function, since $y \in X$.  Also, according to (\ref{X})-(\ref{op:T}) and (\ref{cotadef}),  it turns out that $T_\gamma(y)\in K$.

Through the substitution $v=y'$, the ordinary differential equation  (\ref{eq:y}) is equivalent to
\begin{equation*}
-\dfrac{\Psi_y'(x)+2x}{\Psi_y(x)}=\dfrac{v'(x)}{v(x)},
\end{equation*}
from where we get:
\begin{equation*}
y(x)=y(0)+c_0 \displaystyle\int_0^{x} f_y(\eta) \mathrm{d}\eta.
\end{equation*}

Then, condition (\ref{cond:0}) is satisfied if and only if $c_0=\gamma y(0)$. In addition, due  to (\ref{cond:infty}) we obtain

\begin{equation}\label{alpha}
y(0)=\left(1+\gamma\displaystyle\int_0^{\infty} f_y(\eta) \mathrm{d}\eta \right)^{-1}.
\end{equation}

Therefore, $y$ is a solution to problem (\ref{Pb:c}) if and only if $y$ is a  fixed point of the operator $T_\gamma$, i.e. $y(x)=T_{\gamma}(y)(x), \forall x\geq 0.$
Conversely, if $y$ is a fixed point of the operator $T_{\gamma}$ we obtain immediately that (\ref{cond:infty}) is verified and $y(0)$ is given by (\ref{alpha}). Then, by differentiation (\ref{eq:y}) and (\ref{cond:0}) hold, and then $y$ is a solution to problem $(\ref{Pb:c})$.

\end{proof}

\begin{obs}
The notation $y_\gamma$, $T_{\gamma}$ is adopted in order to emphasize the dependence  of the solution to problem (\ref{Pb:c}) on $\gamma$, although it also depends on $p$ and $\delta$. This fact is going to facilitate the subsequent analysis of the asymptotic behaviour of $y_\gamma$ when $\gamma\to\infty$, to be presented in Section \ref{Sec:Lim}.
\end{obs}

In virtue of Theorem \ref{Th:caract1}, we will focus on proving that $T_\gamma$ is a contracting mapping  on $K$. For that purpose,   we need the following lemmas:

\begin{lema}\label{Le:cotas} Let $y_1, y_2\in K$, $\delta\geq 0, \gamma> 0$, $p\geq 1$ and $x\geq 0$. Then, the following estimates hold:

\begin{enumerate}
\item [a)] $\frac{\sqrt{\pi}}{2(1+\delta)} \leq \left \vert \int\limits_0^\infty f_{y_1}(\eta) \mathrm{d}\eta \right\vert \leq \sqrt{1+\delta} \frac{\sqrt{\pi}}{2}$
\item [b)] $ \left\vert \frac{1}{\Psi_{y_1}(\eta)}-\frac{1}{\Psi_{y_2}(\eta)}\right\vert\leq \delta p \Vert  y_1-y_2\Vert_{\infty}$\\
\item [c)] $\left\vert \exp\left(\int\limits_0^\eta\tfrac{-2\xi}{\Psi_{y_1}(\xi)}\mathrm{d}\xi\right)-\exp\left(\int\limits_0^\eta\tfrac{-2\xi}{\Psi_{y_2}(\xi)}\mathrm{d}\xi\right)\right\vert\leq \frac{2 \delta p \eta^2 }{ \exp\left( \tfrac{\eta^2}{1+\delta}\right)} \Vert  y_1-y_2\Vert_{\infty} $\\
\item [d)]$\displaystyle\int_0^x \left| f_{y_1}(\eta)-f_{y_2}(\eta)\right| \mathrm{d}\eta \leq \dfrac{\sqrt{\pi}}{2}  \delta p \sqrt{1+\delta} (2+\delta) ||y_1-y_2||_{\infty} $ \\
\item [e)] $\left|\frac{1}{1+ \gamma\int\limits_0^\infty  f_{y_1}(\eta)\mathrm{d}\eta}- \frac{1}{1+ \gamma \int\limits_0^\infty  f_{y_2}(\eta)\mathrm{d}\eta}\right| \leq \dfrac{2(1+\delta)^{5/2}}{\gamma\sqrt{\pi}}  \delta p (2+\delta) ||y_1-y_2||_{\infty} $
\end{enumerate}

\end{lema}

\begin{proof}
We follow the method developed in \cite{CeSaTa17}.

\begin{enumerate}
\item[a)] Is is immediate from integrating (\ref{cotafh}) in $\left(0,+\infty \right)$.
\item[b)] Notice first that from the Mean Value Theorem applied to the function \mbox{$r(x)=x^p$} and the fact that \mbox{$1\leq \Psi_y(x)\leq 1+\delta$} for all $y\in K$, we obtain:
\begin{equation*}
\left\vert \frac{1}{\Psi_{y_1}(\eta)}-\frac{1}{\Psi_{y_2}(\eta)}\right\vert\leq \delta \left\lvert y_2^p(\eta)-y_1^p(\eta)\right\rvert\leq \delta p \Vert  y_2-y_1\Vert_{\infty}~.
\end{equation*}
\item[c)] Applying the Mean Value Theorem  to $r(x)=\exp(-2x)$ we have
\begin{eqnarray*}
&&\left\vert \exp\left(\int_0^\eta\frac{-2\xi}{\Psi_{y_1}(\xi)}\mathrm{d}\xi\right)-\exp\left(\int_0^\eta\frac{-2\xi}{\Psi_{y_2}(\xi)}\mathrm{d}\xi\right)\right\vert  \nonumber\\
&&\leq 2 \exp\left( -\frac{\eta^2}{1+\delta}\right) \displaystyle\int_0^\eta \left\vert\frac{\xi}{\Psi_{y_1}(\xi)}-\frac{\xi}{\Psi_{y_2}(\xi)}\right\vert\mathrm{d}\xi \\
&& \leq 2 \exp\left( -\frac{\eta^2}{1+\delta}\right)\eta \displaystyle\int_0^\eta \left\vert  \frac{1}{\Psi_{y_1}(\xi)}-\frac{1}{\Psi_{y_2}(\xi)}  \right\vert \mathrm{d}\xi~.
\end{eqnarray*}
Taking into account item b) we obtain the corresponding estimate.
\item[d)] From items b) and c) we get\\

$\int_0^x \left| f_{y_1}(\eta)-f_{y_2}(\eta)\right| \mathrm{d}\eta \\[0.25cm]
\leq \displaystyle \int_0^x \left\lbrace  \left\vert f_{y_1}(\eta)-\tfrac{\exp\left(-2\int_0^x\tfrac{\xi}{\Psi_{y_2}(\xi)}\mathrm{d}\xi\right)}{\Psi_{y_1}(\eta)}  \right\vert \right.
+\left. \left\vert \tfrac{\exp\left(-2\int_0^x\tfrac{\xi}{\Psi_{y_2}(\xi)}\mathrm{d}\xi\right)}{\Psi_{y_1}(\eta)}-f_{y_{2}}(\eta)  \right\vert\right\rbrace \dee \eta \\[0.25cm]
\leq \displaystyle\int\limits_0^x \left\lbrace    \frac{1}{\Psi_{y_1}(\eta)}\left\vert \exp\left(\int_0^\eta\tfrac{-2\xi}{\Psi_{y_1}(\xi)}\mathrm{d}\xi\right)-\exp\left(\int_0^\eta\tfrac{-2\xi}{\Psi_{y_2}(\xi)}\mathrm{d}\xi\right)\right\vert      \right.\\[0.25cm]
\qquad \qquad +\left. \exp\left(\int_0^\eta\tfrac{-2\xi}{\Psi_{y_2}(\xi)}\mathrm{d}\xi\right) \left\vert \frac{1}{\Psi_{y_1}(\eta)}-\frac{1}{\Psi_{y_2}(\eta)}\right\vert \right\rbrace \dee \eta \\[0.25cm]
\leq  \Vert y_1-y_2\Vert_{\infty}  \delta p \int_0^x \exp\left( \tfrac{-\eta^2}{1+\delta}\right) (2\eta^2+1) \dee \eta \\[0.25cm]
=  \Vert y_1-y_2\Vert_{\infty} \delta p\sqrt{1+\delta} \left[ \frac{\sqrt{\pi}}{2} (2+\delta) \erf\left( \tfrac{x}{\sqrt{1+\delta}}\right)-x\sqrt{1+\delta} \exp\left( \tfrac{-x^2}{1+\delta}\right)\right]\\[0.25cm]
\leq \frac{\sqrt{\pi}}{2} \delta p \sqrt{1+\delta} (2+\delta)  \Vert y_1-y_2\Vert_{\infty}.$\\[0.25cm]

\item[e)] It follows immediately by using (\ref{cotadef}) and item d).
\end{enumerate}
\end{proof}

\begin{lema}\label{Le:caractg}
Let   $\gamma>0$, $p\geq 1$ and
\begin{equation*}\label{def:g}
g_\gamma(x)=xp(1+x)^{3/2}\left[(2+x)\left(1+(1+x)^{3/2}\right)+\frac{2}{\gamma \sqrt{\pi}}(1+x)\right], \quad x\geq 0,
\end{equation*}
then there exist a unique $\delta_\gamma>0$ such that $g_\gamma\left(\delta_\gamma\right)=1$.
\end{lema}

\begin{proof}
It follows  immediately from the fact that $g_\gamma$ is an increasing function, $g_\gamma(0)=0$ and $\lim\limits_{x\to\infty} g_\gamma(x)=+\infty$.
\end{proof}

Now, we are able to formulate the following result:

\begin{teo}\label{Th:contracC}
Let  $\gamma>0$ and $p\geq 1$. The problem (\ref{Pb:c}) has a unique solution $y_\gamma \in K$ if and only if $0\leq \delta < \delta_\gamma$ where $\delta_\gamma$ is given by Lemma \ref{Le:caractg}.
Moreover, $y_{\gamma}$ is a $C^{\infty}$ function in $\mathbb{R}^+$ with the properties
\begin{equation}\label{beta}
y'_{\gamma}(x)>0, \qquad y''_{\gamma}(x)<0, \qquad\forall x\geq 0.
\end{equation}
\end{teo}
\begin{proof}
Let be now $y_1$, $y_2 \in K$ and $x \geq 0$. Taking into account  Lemma \ref{Le:cotas}, we have:\\[-0.2cm]
\begin{align*}
&\left\vert T_\gamma (y_1)(x)-T_\gamma (y_2)(x)\right\rvert\leq \left\vert \dfrac{1+\gamma \int_0^x f_{y_1}(\eta)\dee \eta }{1+\gamma \int_0^\infty f_{y_1}(\eta)\dee\eta }-\dfrac{1+\gamma \int_0^x f_{y_2}(\eta)\dee\eta }{1+\gamma \int_0^\infty f_{y_1}(\eta)\dee\eta }\right\vert\\[0.25cm]
& +\left\vert \dfrac{1+\gamma \int_0^x f_{y_2}(\eta)\dee\eta }{1+\gamma \int_0^\infty f_{y_1}(\eta)\dee\eta }-\dfrac{1+\gamma \int_0^x f_{y_2}(\eta)\dee\eta }{1+\gamma \int_0^\infty f_{y_2}(\eta)\dee\eta }\right\vert\\
& \leq \dfrac{\gamma \displaystyle\int_0^x \left\lvert f_{y_1}(\eta)-f_{y_2}(\eta)\right\rvert \dee \eta}{\left\lvert 1+\gamma \int\limits_0^{\infty} f_{y_1}(\eta)\dee \eta \right\rvert}\\[0.25cm]
&+\left \lvert 1+\gamma \displaystyle\int_0^x f_{y_2}(\eta)\dee \eta \right\rvert \left|\frac{1}{1+ \gamma\int\limits_0^\infty  f_{y_1}(\eta)\mathrm{d}\eta}- \frac{1}{1+ \gamma \int\limits_0^\infty  f_{y_2}(\eta)\mathrm{d}\eta}\right|\\[0.25cm]
&\leq g_\gamma(\delta) \Vert y_1-y_2\Vert_{\infty}.\\
\end{align*}

Then from Lemma \ref{Le:caractg}, if $0 \leq \delta < \delta_\gamma$ it follows that $T_\gamma$ is a contracting mapping what allows to apply the Banach fixed point  theorem. Therefore, the problem (\ref{Pb:c}) has a unique non-negative continuous solution. Moreover, by differentiation and easy computation the solution is a $C^{\infty}$ function in $\mathbb{R}^+$ with the useful properties (\ref{beta}).
\end{proof}

\section{Asymptotic behaviour of p-GME function when $\gamma\to\infty$} \label{Sec:Lim}

In this section if we consider the Stefan problem (\ref{EcCalor})-(\ref{FrontInicial}) and we change the Robin condition (\ref{CondBordeConv}) by a Dirichlet one:
\begin{equation}
T(0,t)=T_0>0,
\end{equation}
we obtain the following ordinary differential problem:
\begin{subequations}\label{Pb:t}
\begin{align}
\label{eq:yt}&[(1+\delta y^p(x))y'(x)]'+2xy'(x)=0\quad 0<x<+\infty\\
\label{cond:0t}&y(0)=0\\
\label{cond:inftyt}&y(+\infty)=1
\end{align}
\end{subequations}

For the special case $p=1$, the solution to this problem is called {\em modified error function} (ME) and was considered in \cite{Bo17,CeSaTa17,CeSaTa18,ChSu74,OlSu87}. In \cite{CeSaTa17}  the existence and uniqueness of the ME function was proved.
Moreover, if it is considered $\delta=0$ the  classical error function  defined by:
\begin{equation}\label{erf}
\erf(x)=\frac{2}{\sqrt{\pi}}\displaystyle\int_0^x\exp(-z^2)dz,\quad x>0,
\end{equation}
arises as a solution.

In a similar way to the above section we can analyse the existence and uniqueness of the  {\em p-modified error function} (p-ME), which is defined as the solution to problem (\ref{Pb:t}) and constitutes a
generalization of the ME function.

Now, let us define
$$K^*=\left\{h\in X\,/\, ||h||_{\infty}\leq 1,\,0 \leq h,\, h(0)=0, \, h(+\infty)=1\right\},$$ where $X$ is given by  (\ref{X}). We remark that $K^*$ is a non-empty closed convex and bounded subset of the Banach space $X$. We will show that the ordinary differential problem (\ref{Pb:t}) becomes equivalent to an integral equation.

\begin{teo}\label{Th:caractT*}
Let  $\delta\geq 0$, $p\geq 1$. The function $y^*\in K^{*}$ is a solution to problem (\ref{Pb:t}) if and only if $y^*$  is a fixed point of the  operator $T^{*}: K^{*} \rightarrow K^{*}$  given  by:
\begin{equation}\label{op:T*}
T^*(h)(x)=\frac{\displaystyle\int_0^{x} f_h(\eta) \mathrm{d}\eta}{ \displaystyle\int_0^{\infty} f_h(\eta) \mathrm{d}\eta},\qquad x\geq 0,
\end{equation}
with $f_h$  defined by (\ref{fh}).
\end{teo}

\begin{proof}
In a similar way that in the proof of Theorem \ref{Th:caract1}, the operator $T^*$ is well defined and it is easy to see that
\begin{equation*}
y^*(x)=y^*(0)+c^*_0 \displaystyle\int_0^{x} f_y^*(\eta) \mathrm{d}\eta,
\end{equation*}
with $y^*(0)=0$ and $c^*_0=\left(\displaystyle\int_0^{\infty} f_h(\eta) \mathrm{d}\eta\right)^{-1}$.

Then, using (\ref{cond:0t}) and (\ref{cond:inftyt}), we obtain (\ref{op:T*}). Therefore, $y^*$ is a solution to the problem (\ref{Pb:t}) if and only if $y^*$ is a fixed point of the operator $T^*$.
\end{proof}

In order to prove that the operator $T^*$ is a contracting mapping on $K^*$, we enunciate the following lemmas which proofs are analogous as Lemma \ref{Le:cotas} and Lemma \ref{Le:caractg}.

\begin{lema}\label{Le:cotas*}
Let $y^*_1, y^*_2\in K^*$, $\delta\geq 0$, $p\geq 1$ and $x\geq 0$. Then, the following estimate holds: $$\left|\frac{1}{\int\limits_0^\infty  f_{y^*_1}(\eta)\mathrm{d}\eta}- \frac{1}{ \int\limits_0^\infty  f_{y^*_2}(\eta)\mathrm{d}\eta}\right| \leq \dfrac{2(1+\delta)^{5/2}}{\sqrt{\pi}}  \delta p (2+\delta) ||y^*_1-y^*_2||_{\infty}.$$
\end{lema}

\begin{lema}\label{Le:caractg*}
Let    $p\geq 1$ and
\begin{equation*}\label{def:g*}
g^*(x)=xp(1+x)^{3/2}(2+x)\left(1+(1+x)^{3/2}\right), \quad x\geq 0,
\end{equation*}
then there exist a unique $\delta^*>0$ such that $g^*\left(\delta^*\right)=1$.
\end{lema}

\begin{teo}\label{Th:contracT}
The problem (\ref{Pb:t}) has a unique solution $y^* \in K$ if and only if  $0\leq \delta < \delta^*$ where $\delta^*$ is given by Lemma \ref{Le:caractg*}. Moreover, $y^{*}$ is a $C^{\infty}$ function in $\mathbb{R}^+$.
\end{teo}
\begin{proof}
Let be now $y^*_1$, $y^*_2 \in K^*$ and $x \geq 0$. Taking into account  Lemma \ref{Le:cotas} and Lemma \ref{Le:cotas*} we can obtain:\\[-0.2cm]
\begin{equation*}
\begin{split}
&\left\vert T^*(y^*_1)(x)-T^*(y^*_2)(x)\right\rvert\leq \left\vert \dfrac{\int_0^x f_{y^*_1}(\eta)\dee \eta }{ \int_0^\infty f_{y^*_1}(\eta)\dee\eta }-\dfrac{ \int_0^x f_{y^*_2}(\eta)\dee\eta }{\int_0^\infty f_{y^*_1}(\eta)\dee\eta }\right\vert\\[0.25cm]
&+\left\vert \dfrac{\int_0^x f_{y^*_2}(\eta)\dee\eta }{\int_0^\infty f_{y^*_1}(\eta)\dee\eta }-\dfrac{\int_0^x f_{y^*_2}(\eta)\dee\eta }{\int_0^\infty f_{y^*_2}(\eta)\dee\eta }\right\vert\\[0.25cm]
& \leq \dfrac{\displaystyle\int_0^x \left\lvert f_{y^*_1}(\eta)-f_{y^*_2}(\eta)\right\rvert \dee \eta}{\left\lvert \int\limits_0^{\infty} f_{y^*_1}(\eta)\dee \eta \right\rvert}\\[0.25cm]
& +\left \lvert  \displaystyle\int_0^x f_{y^*_2}(\eta)\dee \eta \right\rvert \left|\frac{1}{\int\limits_0^\infty  f_{y^*_1}(\eta)\mathrm{d}\eta}- \frac{1}{\int\limits_0^\infty  f_{y^*_2}(\eta)\mathrm{d}\eta}\right|\\[0.25cm]
& \leq g^*(\delta^*) \Vert y^*_1-y^*_2\Vert_{\infty}.\\[0.25cm]
\end{split}
\end{equation*}

Then from Lemma \ref{Le:caractg*}, if $0 \leq \delta < \delta^*$ it follows that $T^*$ is a contracting mapping what allows to apply the Banach fixed point  theorem. Therefore, the problem (\ref{Pb:t}) has a unique non-negative continuous solution which is also a $C^{\infty}$ function by simple differentiation in $\mathbb{R}^+$.
\end{proof}

In problem (\ref{Pb:c}), a Robin boundary condition    characterized by the coefficient $\gamma>0$ at $x = 0$ is imposed. This condition constitutes a generalization of the Dirichlet one  in the sense that if we take the limit when $\gamma \rightarrow \infty$ in condition (\ref{cond:0}) we  obtain condition  (\ref{cond:0t}). Now, we will show that the solution to problem (\ref{Pb:c}) converges to the solution to problem (\ref{Pb:t}) when  $\gamma \rightarrow \infty$. For that purpose, first, we need the following lemmas which proofs are immediate.

\begin{lema} For every $p\geq 1$, when $\gamma\to\infty$, the following convergence results hold:
\begin{enumerate}
\item[a)] $T_\gamma(h)(x) \to T^*(h)(x)$ for every $h\in K$ and $x\geq 0$.
\item[b)] $g_\gamma(x)\to g^*(x)$ for every $x\geq 0$.
\item[c)] $\delta_\gamma \to \delta^*$.
\end{enumerate}
In addition $g_\gamma(x)\geq g^*(x)$ and $\delta_\gamma<\delta^*$ for all $x\geq 0, \gamma>0$.
\end{lema}

\begin{lema}\label{Le:caractC}
Let $p\geq 1$ and
\begin{equation}\label{Def: C}
C(x)=2x p (1+x)^3  (2+x), \qquad x\geq 0.
\end{equation}
then    there exists a unique $\hat{\delta}>0$ such that $C(\hat{\delta})=1$.
\end{lema}

\begin{teo}\label{Th:convergence}
Let $p\geq 1$ and  $0\leq \delta<\min\lbrace\hat{\delta},\delta_\gamma\rbrace$ . Then
$\vert \vert y_\gamma -y^*\vert\vert_\infty \to 0 $ when $\gamma \to\infty$. Furthermore  the order of convergence is $\frac{1}{\gamma}$ when $\gamma\to\infty$.
\end{teo}
\begin{proof}

First let us note that if $0\leq \delta<\min\lbrace\hat{\delta},\delta_\gamma\rbrace$ then as $\delta_\gamma<\delta^*$, we obtain that $y_\gamma$ and $y^*$ are well defined because of Theorems \ref{Th:contracC} and \ref{Th:contracT}.
Then for $x\geq 0$ we have:
\bigskip

$\vert y_\gamma (x)-y^*(x)\vert=\left\lvert\frac{ \left( 1+\gamma \int_0^x f_{y_\gamma}(\eta) \mathrm{d}\eta \right) \left( \int_0^\infty f_{y^*}(\eta) \mathrm{d}\eta \right)  -\left(  \int_0^x f_{y^*}(\eta) \mathrm{d}\eta \right)\left( 1+\gamma \int_0^\infty f_{y_\gamma}(\eta) \mathrm{d}\eta \right)  }{\left( 1+\gamma \int_0^\infty f_{y_\gamma}(\eta) \mathrm{d}\eta \right) \left( \int_0^\infty f_{y^*}(\eta) \mathrm{d}\eta \right)}\right\rvert\\[0.3cm]
\leq \left\lvert\frac{ \int_0^\infty f_{y^*}(\eta) \mathrm{d}\eta+\gamma\left( \int_0^x f_{y_\gamma}(\eta) \mathrm{d}\eta \right) \left( \int_0^\infty f_{y^*}(\eta) \mathrm{d}\eta \right)  - \int_0^x f_{y^*}(\eta) \mathrm{d}\eta -\gamma\left(  \int_0^x f_{y^*}(\eta) \mathrm{d}\eta \right)\left(  \int_0^\infty f_{y_\gamma}(\eta) \mathrm{d}\eta \right)  }{\left( 1+\gamma \int_0^\infty f_{y_\gamma}(\eta) \mathrm{d}\eta \right) \left( \int_0^\infty f_{y^*}(\eta) \mathrm{d}\eta \right)}\right\rvert\\[0.3cm]
=\left\lvert\frac{ \int_x^\infty f_{y^*}(\eta) \mathrm{d}\eta+\gamma\left( \int_0^x f_{y_\gamma}(\eta) \mathrm{d}\eta \right) \left( \int_0^\infty f_{y^*}(\eta) \mathrm{d}\eta \right)  - \gamma\left(  \int_0^x f_{y^*}(\eta) \mathrm{d}\eta \right)\left(  \int_0^\infty f_{y_\gamma}(\eta) \mathrm{d}\eta \right)  }{\left( 1+\gamma \int_0^\infty f_{y_\gamma}(\eta) \mathrm{d}\eta \right) \left( \int_0^\infty f_{y^*}(\eta) \mathrm{d}\eta \right)}\right\rvert\\[0.3cm]
\leq  \left\lvert\frac{ \int_0^\infty f_{y^*}(\eta) \mathrm{d}\eta+\gamma\left( \int_0^x f_{y_\gamma}(\eta) \mathrm{d}\eta \right) \left( \int_0^\infty f_{y^*}(\eta) \mathrm{d}\eta-\int_0^\infty f_{y_\gamma}(\eta) \mathrm{d}\eta  \right) + \gamma\left( \int_0^\infty f_{y_\gamma}(\eta) \mathrm{d}\eta - \int_0^x f_{y^*}(\eta) \mathrm{d}\eta \right)\left(  \int_0^\infty f_{y_\gamma}(\eta) \mathrm{d}\eta \right)  }{\left( 1+\gamma \int_0^\infty f_{y_\gamma}(\eta) \mathrm{d}\eta \right) \left( \int_0^\infty f_{y^*}(\eta) \mathrm{d}\eta \right)}\right\rvert\\[0.3cm]
\leq \frac{ \sqrt{1+\delta}\frac{\sqrt{\pi}}{2}+\gamma\sqrt{1+\delta}\frac{\sqrt{\pi}}{2} \left( \int_0^\infty \lvert f_{y^*}(\eta)-f_{y_\gamma}(\eta)\rvert\mathrm{d}\eta \right) + \gamma\sqrt{1+\delta}\frac{\sqrt{\pi}}{2}\left( \int_0^x \lvert f_{y^*}(\eta)-f_{y_\gamma}(\eta)\rvert\mathrm{d}\eta \right) }{\frac{\gamma \sqrt{\pi}}{2(1+\delta)}\frac{\sqrt{\pi}}{2(1+\delta)}}\\[0.3cm]
\leq  \frac{ \sqrt{1+\delta}\frac{\sqrt{\pi}}{2}+2\gamma\sqrt{1+\delta}\frac{\sqrt{\pi}}{2} \int_0^\infty \lvert f_{y^*}(\eta)-f_{y_\gamma}(\eta)\rvert\mathrm{d}\eta }{\frac{\gamma \pi}{4(1+\delta)^2}}\\[0.3cm]
\leq  \frac{4(1+\delta)^{2}}{\gamma \pi}\left(\sqrt{1+\delta}\frac{\sqrt{\pi}}{2}+\gamma \frac{\pi}{4} \delta p (1+\delta) (2+\delta) \vert\vert y_\gamma-y^*\vert\vert_\infty\right)\\[0.3cm]
\leq  \frac{2(1+\delta)^{5/2}}{\gamma \pi}+2(1+\delta)^3 \delta p(2+\delta) \vert\vert y_\gamma-y^*\vert\vert_\infty~.$

\bigskip
Note that these inequalities, which are obtained by applying Lemma \ref{Le:cotas}, leads to
\begin{equation*}
\left(1-C(\delta)\right)\vert\vert y_\gamma -y^*\vert \vert_\infty \leq \frac{1}{\gamma} \left(\frac{2(1+\delta)^{5/2}}{\sqrt{\pi} }\right)
\end{equation*}
with $C$ defined by (\ref{Def: C}). Finally, the desired convergence and order of convergence of Theorem \ref{Th:convergence} are obtained by noting that if  $0\leq \delta<\hat{\delta}$, then $0\leq C(\delta)<1$ due to Lemma \ref{Le:caractC}.
\end{proof}

\section{Existence and uniqueness of solution to an ordinary differential problem that arises when considering a Neumann condition}

In this section we will consider a solidification Stefan problem with a Neumann condition at the fixed face $x=0$ given by:

\begin{align}
& \rho c \frac{\partial T}{\partial t}=\frac{\partial}{\partial x} \left(k(T)\frac{\partial T}{\partial x}\right),& 0<x<s(t), \quad t>0, \label{EcCalorFlujo}\\
& k(T(0,t)) \frac{\partial T}{\partial x}(0,t)=\frac{q_0}{\sqrt{t}}, &t>0, \label{CondBordeFlujo}\\
&  T(s(t),t)=T_{_f}, &t>0, \label{TempCambioFaseFlujo}\\
&  k\left(T(s(t),t)\right)\frac{\partial T}{\partial x}(s(t),t)=\rho l \dot s(t), &t>0, \label{CondStefanFlujo}\\
& s(0)=0,\label{FrontInicialFlujo}
\end{align}
where the unknown functions are the temperature $T$ and the free boundary $s$ separating both phases. The parameters $\rho>0$ (density), $l>0$ (latent heat per unit mass), $T_f$ (phase-change temperature), $q_0>0$   (characterizes the heat flux on the fixed face $x=0$ of the face-change material which can be determined experimentally) and $c>0$ (specific heat) are all known constants. In this case, the thermal conductivity $k$ is defined as:
\begin{align}
&k(T)=k_{0}\left(1+\delta\left(\frac{T}{T_f}\right)^p\right), \quad p \geq 1 \label{kflujo},
\end{align}
where $\delta$ is a given positive constant and $k_{0}$ is the reference thermal conductivity.

In a similar way as in previous sections, this Stefan problem leads us to study the following ordinary differential problem:
\begin{subequations}\label{Pb:Flujo}
\begin{align}
\label{eq:yf}&[(1+\delta y^p(x))y'(x)]'+2xy'(x)=0\quad 0<x<+\infty\\
\label{cond:flujo}&\left(1+\delta y^p(0)\right)y'(0)=\gamma^*\\
\label{cond:inftyf}&y(+\infty)=1
\end{align}
\end{subequations}
where
\begin{equation}
\gamma^*=2\mathrm{Bi}^*\qquad
 \mathrm{with\qquad Bi}^*=\frac{q_0 \sqrt{\alpha_0}}{k_0 T_f}.
\end{equation}

In a similar way to the above sections we can state the following results:
\begin{teo}\label{Th:caractgamma*}
Let  $\delta\geq 0$, , $p\geq 1$ and $0<\gamma^*\leq \tfrac{2}{\sqrt{\pi(1+\delta)}}$. The function $y_{\gamma^*} \in K$ is a solution to problem (\ref{Pb:Flujo}) if and only if $y_{\gamma^*}$  is a fixed point of the  operator $T_{\gamma^*}: K \rightarrow K $  given  by:
\begin{equation}\label{op:Tgamma*}
T_{\gamma^*}(h)(x)=1-\gamma^* \displaystyle\int_x^{+\infty} f_h(\eta) \mathrm{d}\eta,\qquad x\geq 0,
\end{equation}
with $f_h$ defined by (\ref{fh}) and $K$ given by (\ref{K}).
\end{teo}
\begin{proof}
Given $y_{\gamma^{*}} \in K$ and taking into account (\ref{cotafh}), we obtain
\begin{equation}\label{cotadefflujo}
0<\dfrac{\gamma^* \erfc(x)}{1+\delta} \leq  \gamma^* \displaystyle\int_x^{\infty} f_y(\eta) \mathrm{d}\eta <  \dfrac{\gamma^* \sqrt{1+\delta} \sqrt{\pi}}{2} \leq 1.
\end{equation}

Notice that from (\ref{op:Tgamma*}) we have that $T_{\gamma^*}(y_{\gamma^{*}})$ in an analytic function, since $y_{\gamma^{*}} \in X$. Also, according to (\ref{op:Tgamma*})  and (\ref{cotadefflujo}), it turns out that $T_{\gamma^*}(y_{\gamma^{*}}) \in K$.

In a similar way as in the proof of Theorem \ref{Th:caract1}, $y_{\gamma^{*}} $ is a solution to problem (\ref{Pb:Flujo}) if and only if $y_{\gamma^{*}} $ is a fixed point of the operator $T_{\gamma^*}$.

\end{proof}

\begin{teo}\label{unicidadNeumann}
Let $p \geq 1$, $\delta >0$ and $0<\gamma^*\leq \tfrac{2}{\sqrt{\pi(1+\delta)}}$.

The problem (\ref{Pb:Flujo}) has a unique $C^{\infty}$ solution $y_{\gamma^*} \in K$ if and only if $\delta < \delta_{\gamma^*}$ where $\delta_{\gamma^*}$ is the unique solution to the equation $$g(x)=1,$$ with $$g(x)= x \frac{p}{\sqrt{\pi}} \left[(1+x)\left(\sqrt{1+x}\exp(-\tfrac{1}{4})+\sqrt{\pi}\right)+\sqrt{\pi}\right].$$
\end{teo}
\begin{proof}

Let be $y_{1_{\gamma^*}}$, $y_{2_{\gamma^*}} \in K$ and $x \geq 0$. Taking into account properties b) and c) of Lemma \ref{Le:cotas} we can obtain:\\

$\left\vert T_{\gamma^*}(y_{1_{\gamma^*}})(x)-T_{\gamma^*}(y_{2_{\gamma^*}})(x)\right\rvert$\\[0.3cm]
$\leq  \Vert y_1-y_2\Vert_{\infty} \;\delta p \gamma^* \left[ (1+\delta)^{3/2} \left( \tfrac{x}{\sqrt{1+\delta}} \exp\left(-\frac{x^2}{1+\delta}\right)+ \frac{\sqrt{\pi}}{2}\right)+\sqrt{1+\delta} \frac{\sqrt{\pi}}{2}\right]\\[0.3cm]
\leq  \delta \frac{p}{\sqrt{\pi}}  \left[(1+\delta) \left(\sqrt{1+\delta}\exp(-\tfrac{1}{4})+\sqrt{\pi}\right)+\sqrt{\pi}\right)] \Vert y_{1_{\gamma^*}}-y_{2_{\gamma^*}}\Vert_{\infty}\\[0.3cm]
 \leq g(\delta) \Vert y_{1_{\gamma^*}}-y_{2_{\gamma^*}}\Vert_{\infty}.$\\

Since $g$ is an increasing function such that $g(0)=0$ and $g(+\infty)=+\infty$, then there exists a unique   $\delta_{\gamma^*}>0$ with $g( \delta_{\gamma^*})=1$.

 Then, if $0 \leq \delta <  \delta_{\gamma^*}$ it follows that $T_{\gamma^*}$ is a contracting mapping what allows to apply the Banach fixed point  theorem. Therefore, the problem (\ref{Pb:Flujo}) has a unique non-negative continuous solution which is also a $C^{\infty}$ function.

\end{proof}

\section{Conclusion}
In this article, the ordinary differential problems studied in \cite{CeSaTa17, CeSaTa18} have been generalized by defining what we call the  p-GME function and the \mbox{p-ME function}  corresponding to the case when a Robin  or Dirichlet boundary condition are imposed at  $x=0$, respectively. In both problems, existence and uniqueness of $C^{\infty}$ solution has been proved by defining  convenient contracting mappings . In addition it has been studied the behaviour of the p-GME function when the coefficient $\gamma$ that characterizes the Robin condition goes to infinity, obtaining its convergence to the p-ME function with an  order of convergence of the type $1/\gamma$ when $\gamma\to\infty$. Finally, existence and uniqueness of solution to a solidification problem with a Neumann condition has also been studied.

\section*{Acknowledgments}

The author gratefully acknowledges the many helpful suggestions of the
reviewer to improve the paper.

The present work has been partially sponsored by European Union's Horizon 2020
research and innovation programme under the Marie Sklodowska-Curie
Grant Agreement No. 823731 CONMECH, and by the Project PIP No 0275 from CONICET - Univ. Austral, Rosario, Argentina  and ANPCyT PICTO Austral No 0090.

\small{
}

\end{document}